\documentclass[12pt]{amsart}
\usepackage{amssymb, amsmath, mathtools}

\usepackage[dvipsnames]{xcolor}
\usepackage{color}
\usepackage[utf8]{inputenc}
\usepackage[T1]{fontenc}
\usepackage{fullpage}
\usepackage[normalem]{ulem}
\usepackage{amsmath,amscd,mathrsfs}
\usepackage{tikz}
\usepackage[colorlinks=true, citecolor=blue, linkcolor=red,pagebackref, hyperindex]{hyperref}
\usepackage{mathtools}

\newtheorem{Theorem}{Theorem}[section]

\newtheorem{Potential Theorem}[Theorem]{Potential Theorem}
\newtheorem{Lemma}[Theorem]{Lemma}
\newtheorem{Corollary}[Theorem]{Corollary}
\newtheorem{Proposition}[Theorem]{Proposition}
\theoremstyle{definition}
\newtheorem{Example}[Theorem]{Example}

\newtheorem{Definition}[Theorem]{Definition}
\newtheorem{Question}[Theorem]{Question}

\theoremstyle{remark}

\DeclareMathOperator{\depth}{depth}

\DeclareMathOperator{\length}{\lambda}

\DeclareMathOperator{\rank}{rank}

\DeclareMathOperator{\frk}{frk}

\DeclareMathOperator{\ehk}{e_{HK}}
\DeclareMathOperator{\eh}{e}
\DeclareMathOperator{\fpt}{fpt}

\DeclareMathOperator{\s}{s}
\DeclareMathOperator{\Trace}{Tr}
\DeclareMathOperator{\Dis}{D}
\DeclareMathOperator{\Gr}{Gr}
\DeclareMathOperator{\lc}{H}

\def\m{\mathfrak{m}}

\def\R{\mathbb{R}}

\def\N{\mathbb{N}}

\renewcommand{\geq}{\geqslant} 
\renewcommand{\leq}{\leqslant}

\newcommand{\ul}{\underline}

\newcommand{\pow}[2]{(#1)^{\oplus #2}}

\begin{document}

\title{Continuity of Hilbert--Kunz multiplicity and F-signature}
\author{Thomas Polstra}
\thanks{Polstra was supported in part by NSF Postdoctoral Research Fellowship DMS $\#1703856$.}
\address{Department of Mathematics, University of Utah, Salt Lake City}
\email{polstra@math.utah.edu}
\author{Ilya Smirnov}
\address{Department of Mathematics, University of Michigan, Ann Arbor}
\curraddr{
Department of Mathematics, Stockholm University, SE - 106 91 Stockholm, Sweden
}

\email{ismirnov@umich.edu}

\begin{abstract} We establish the continuity of Hilbert--Kunz multiplicity and F-signature as functions from a Cohen-Macaulay local ring $(R,\m,k)$ of prime characteristic to the real numbers at reduced parameter elements with respect to the $\m$-adic topology.

\end{abstract}


\maketitle

\section{Introduction}

In this paper we consider the problem of determining when two rings have similar singularities. As a motivating example let $R = k[[x_1, \ldots, x_n]]$ be the power series ring in $n$ variables over a field $k$ and $f \in R$. 
If we add a term of sufficiently high order to $f$, will the new singularity be similar to that of $f$?

The origin of the above question can be traced at least as far back as 1956. Samuel (\cite{Samuel1956})
proved if the Jacobian ideal $J(f)$ of $f$ is $\m$-primary and if $f - g \in (x_1, \ldots, x_n)J(f)^2$, then there is an automorphism $\varphi\colon R\to R$ such that $\varphi(f)=g$. In particular, Samuel's result asserts that if $f$ has an isolated singularity and $g$ is sufficiently close to $f$, then the rings $R/(f)$ and $R/(g)$ are isomorphic. Samuel's result was furthered by Hironaka (\cite{Hironaka1965}): if $R/I$ is a reduced and equidimensional isolated singularity, then there is an integer $e > 0$ such that for any ideal $J$ that also defines a reduced and equidimensional singularity of same dimension and if $I \equiv J \mod (x_1, \ldots, x_n)^e$ then $R/I \cong R/J$.

More recently, Cutkosky and Srinivasan have extended Samuel's result to include all complete intersection prime ideals and Hironaka's theorem to to all reduced equidimensional ideals. See (\cite{CutkoskySrinivasan1993}) for precise statements.

However, when $f\in k[[x_1,\ldots,x_n]]$ does not have an isolated singularity, the results of Samuel, Hironaka, and Cutkosky and Srinivasan do not allow us to relate the singularities of $R/(f)$ with the singularities of $R/(g)$. To compare the singularity of $R/(f)$ with the singularity of $R/(g)$ we first require a formal notion of ``closeness'' of singularities. For example, one could require the Hilbert--Samuel functions of $R/(f)$ and $R/(g)$ to coincide. Srinivas and Trivedi 
(\cite{SrinivasTrivedi}) have shown if $I=(f_1,\ldots, f_h)$ is a complete intersection\footnote{This result was recently reproved by Adamus and Patel (\cite{AP}) using a more combinatorial method.}, or more generally if $I$ is a parameter ideal in a generalized Cohen-Macaulay ring, then the associated graded rings of $I$ and an ideal $J=(f_1+\epsilon_1,\ldots, f_h+\epsilon_h)$ will be isomorphic, provided $\epsilon_1,\ldots \epsilon_h$ are in a sufficiently large power of the maximal ideal. 

We now turn our attention to numerical invariants in positive characteristic. In particular, we introduce and explore the behavior of F-pure threshold, Hilbert--Kunz multiplicity, and F-signature with respect to $\m$-adic topologies. In what follows, let $(R, \mathfrak m)$ be a local ring of prime characteristic $p$. Let $F\colon R\to R$ denote the Frobenius endomorphism. Given ideal $I\subseteq R$ and $e\in \N$ let $I^{[p^e]}=(i^{p^e}\mid i\in I)$ be the expansion of $I$ along $F^e$. We shall always assume $R$ is F-finite, meaning that $F$ is a finite map. 

The study of F-pure thresholds was initiated by Takagi and Watanabe (\cite{TakagiWatanabe2004}). If $(R,\m,k)$ is regular then the F-pure threshold of an element $f\in R$ can be defined as the limit 
$\fpt(f)=\lim\limits_{e\to \infty}\frac{\nu_e(f)}{p^e}$ where $\nu_e(f)=\max \{t\in \N\mid f^t\not \in \m^{[p^e]}\}.$ Two elementary observations concerning F-pure thresholds are that if $f,g\in R$ then $\fpt(f+g)\leq \fpt(f)+\fpt(g)$ and that if $f\in \m^{[p^{e_0}]}$ then $\fpt(f)\leq \frac{1}{p^{e_0}}$. From these observations it follows that the F-pure threshold function is continuous as a function from $R\to \R$:  specifically, if $f\in R$ and $\delta>0$ then there exists $N\in \N$ such that for all $g\in R$ such that $f-g\in \m^N$, $|\fpt(f)-\fpt(g)|<\delta$. Continuity of the F-pure threshold function motivates studying continuity properties of Hilbert--Kunz multiplicity and F-signature 
and sheds light onto the problem of relating the singularities of $R/(f)$ with $R/(g)$ when $f$ and $g$ are suitably close.

In 1983, Paul Monsky, building on the earlier work of Ernst Kunz (\cite{Kunz1969,Kunz1976}), defined a new invariant, the Hilbert--Kunz multiplicity of $R$, as the limit
\[
\ehk (R) = \lim_{e \to \infty} \frac{\length (R/\mathfrak m^{[p^e]})}{p^{ed}}.
\]
Monsky (\cite{Monsky1983}) showed that this limit always exists. Values of $\ehk(R)$ dictate the severity of the singularities of $R$. Most notable is that under mild hypotheses $\ehk(R)\geq 1$ with equality if and only if $R$ is regular (\cite{WatanabeYoshida2000}). It has also been shown in (\cite{BlickleEnescu, AberbachEnescu2008}) that small enough values of $\ehk(R)$ imply $R$ is Gorenstein and strongly F-regular. To be able to relate the singularities of $R/(f)$ and $R/(g)$ when $f$ and $g$ are suitably close it is natural to investigate the following question, which was originally asked to the second author by Luis N\'u\~nez-Betancourt:

\begin{Question}\label{hypersurface q}
Let $(R,\m,k)$ be a regular local ring of prime characteristic. Is the Hilbert--Kunz multiplicity function defined by $f \mapsto \ehk (R/(f))$ a continuous function in the $\m$-adic topology? Namely, if $f \in R$ is an element and we are given $\delta > 0$, can we find an integer $N$ such that for all $g \in R$ such that $f - g \in \m^N$,
\[
\left |\ehk (R/(f)) - \ehk(R/(g)) \right| < \delta?
\]
\end{Question}

We now turn our attention to F-signature. Given finitely generated $R$-module $M$ we let $\frk(M)$ denote the largest rank of a free $R$-module $F$ for which there is a surjection $M\to F$. The F-signature of $(R,\m,k)$ was defined by Huneke and Leuschke (\cite{HunekeLeuschke}), conceptualizing the earlier work \cite{SmithVdB} of Smith and Van den Bergh, as the limit 
\[
\s (R) =\lim_{e\to \infty}\frac{\frk(R^{1/p^e})}{\rank(R^{1/p^e})}.
\]

The number $\frk(R^{1/p^e})$ will be denoted by $a_e(R)$ and refered to as the $e$th Frobenius splitting number of $R$. Tucker (\cite{TuckerF-sigExists}) showed existence of the above limit for all local rings.\footnote{The existence of the above limit has been recently established without the local hypothesis in (\cite{DPY}).} F-signature is a measurement of singularities: in particular, $s(R)\leq 1$ with equality if and only if $R$ is regular by (\cite{HunekeLeuschke}) and $s(R)>0$ if and only if $R$ is strongly F-regular by (\cite{AberbachLeuschke}). As with Hilbert--Kunz multiplicity, it is natural to relate the singularities of $R/(f)$ with $R/(g)$ by investigating continuity properties of F-signature with respect to the $\m$-adic topology.

\begin{Question}\label{hypersurface q 2}
Let $(R,\m,k)$ be a regular local ring of prime characteristic. Is the F-signature function defined by $f \mapsto s(R/(f))$ a continuous function with respect to the $\m$-adic topology? Namely, if $f \in R$ is an element and we are given $\delta > 0$, can we find an integer $N$ such that for all $g \in R$ such that $f - g \in \m^N$ 
\[
\left |s(R/(f)) - s(R/(g)) \right| < \delta?
\]
\end{Question}

We are able to answer Question~\ref{hypersurface q}  and Question~\ref{hypersurface q 2},
provided that $f$ has no multiple factors, i.e., $R/(f)$ is reduced. 
Our main contribution is the following theorem:

\begin{Theorem}[{Corollary~\ref{HK continuity} and Theorem~\ref{not main thm}}]\label{Intro theorem}
Let $(R, \m, k)$ be a complete Cohen-Macaulay F-finite ring of prime characteristic $p > 0$. If $(f_1, \ldots, f_c)$ is a parameter ideal such that $R/(f_1, \ldots, f_c)$ is reduced then for any $\delta > 0$ there exists an integer $N > 0$ such that for any $\epsilon_1, \ldots, \epsilon_c \in \m^N$:
\[
|\ehk(R/(f_1, \ldots, f_c)) - \ehk (R/(f_1 + \epsilon_1, \ldots, f_c + \epsilon_c))| < \delta
\]
If we additionally assume that $R$ is Gorenstein then there exists an integer $N > 0$ such that for any $\epsilon_1, \ldots, \epsilon_c \in \m^N$ 
\[
|\s (R/(f_1, \ldots, f_c)) - \s (R/(f_1 + \epsilon_1, \ldots, f_c + \epsilon_c))| < \delta.
\]

\end{Theorem}

We show that Hilbert--Kunz multiplicity and F-signature functions are the uniform limit of functions which are locally constant with respect to the $\m$-adic topology. Naively, one might also expect the Hilbert--Kunz multiplicity and F-signature functions are locally constant with respect to the $\m$-adic topology: 
$\ehk (R/(f)) = \ehk (R/(g))$ and $s(R/(f))=s(R/(g))$ can be forced by making $f - g$ is sufficiently small. 
We point that this cannot hold for Hilbert--Kunz multiplicity in Example~\ref{not constant example} and 
for F-signature in Example~\ref{fsig}.

Similar to a number of recent developments in the study of invariants in prime characteristic rings, the proof of Theorem~\ref{Intro theorem} comes by showing uniform convergence of continuous functions. We get our uniform convergence result by using the methods of Polstra and Tucker (\cite{PolstraTucker}), where it was shown that suitable Noether normalizations provide a kind of ``canonical'' convergence estimate. The essential reason is that when $R$ is a module-finite generically separable extension of a regular local ring $A$, then $R[A^{1/p}] \cong R^{\oplus p^d}$.  The ``canonicity'' in our construction comes from discriminants, and we show that the discriminant of $R/(f)$ does not ``change'' much when $f$ changes slightly (Lemma~\ref{Dis lemma}). This allows us to show that one can control the convergence rate of Hilbert--Kunz function uniformly, independent of a small perturbation (Theorem~\ref{main eHK thm}).

\section{Preliminary Results}

\begin{Definition}
Let $A$ be a ring and $R$ a finite $A$-algebra which is free as an $A$-module. Let $e_1, \ldots, e_n$ be a basis of $R$ as an $A$-module. The discriminant of $R$ is 
\[
\Dis_A(R) = \det \begin{pmatrix} 
\Trace (e_1^2) & \Trace (e_1 e_2) & \cdots & \Trace (e_1e_n) \\
\Trace (e_2e_1) & \Trace (e_2^2) & \cdots & \Trace (e_2e_n)\\
\vdots & \vdots & \cdots & \vdots \\
\Trace (e_ne_1) & \Trace (e_n e_2) & \cdots & \Trace (e_n^2)
\end{pmatrix},
\]
where $\Trace (r)$ denotes the trace of the multiplication map on $R$.
Up to multiplication by a unit of $A$, the discriminant is independent of choice of basis.
\end{Definition}

If $A$ is a regular local ring and $R$ is a finite Cohen-Macaulay extension of $A$, then $R$ is necessarily a free $A$-module by the Auslander-Buchsbaum formula, and  thus the discriminant of $R$ over $A$ is well-defined. We begin by recording some well-known and useful facts concerning discriminants.

\begin{Proposition}\label{Discriminant Properties} Let $A$ be a ring and $R$ a finitely generated $A$-algebra which is free as an $A$-module.
\begin{enumerate}
\item\label{Discriminant Properties Part 1} If $A$ is a normal domain, then $D_A(R)\in A$.
\item\label{Discriminant Properties Part 2} If $I$ is an ideal of $A$ then the image of $D_A(R)$ in $R/IR$ is $D_{A/I}(R/IR)$.
\item\label{Discriminant Properties Part 3} If $A$ is a domain then the ring $R$ is generically separable over $A$ if and only if $D_A(R)\not=0$.
\end{enumerate}
\end{Proposition}

\begin{proof}
We refer the reader to \cite[Page 200]{HochsterFTC} for a proof of (1). For (2), if $e_1,...,e_n\in R$ form a basis of $R$ as an $A$-module then the images of $e_1,...,e_n$ form a basis of $R/IR$ as an $A/IA$-module. Moreover, if $x\in R$ and $\overline{x}$ denotes the image of $x$ in $R/IR$, then $\Trace(\overline{x})=\overline{\Trace(x)}$. Hence the image of $\Dis_A(R)$ in $R/IR$ is indeed $\Dis_{A/I}(R/IR)$.

For (3), as a first step, let $L$ be the fraction field of $A$. We observe that $\Dis_A(R) = \Dis_L (R \otimes_A L)$, since the decomposition of an element $a = \sum a_i e_i$ over $A$ can be considered as a decomposition of $a \otimes 1 = \sum a_i e_i \otimes 1$ of an element $a\otimes 1$ over $L$ and $e_i \otimes 1$ still form a basis of $R \otimes_A L$ over $L$. Thus we can replace $R$ by $R\otimes_A L$ and assume $A = L$ is a field. It is well-known that $\Dis_L(R) \ne 0$ if $R$ is a finite separable field extension of $L$. By linearity this can be extended to the product and show one direction. 

For the other direction, we first show that $R$ is reduced. Since the nilradical is an ideal, it is a vector space over $L$, 
so we may extend a basis of it to a basis of $R$ over $L$.
If $x$ is a nilpotent element, then 
$\Trace (xe) = 0$ for any element $e \in R$, because the trace of a nilpotent matrix is nilpotent. 
This shows that $\Trace$ gives a degenerate bilinear form and $\Dis_L (R) = 0$. Since $R$ is reduced and finite over $L$, it is a direct sum of fields. Again, by linearity 
we can see that each of these fields must be separable over $L$. 
\end{proof}

Let $(R,\m,k)$ be a complete local ring of positive characteristic $p > 0$ with the residue field $k$. 
Cohen's structure theorem shows that there is an injection $\phi \colon k \to R$ such that the induced map to $R/\m$ is an isomorphism. We call $\phi (k)$ a coefficient field. Moreover, Cohen's structure theorem asserts that for any system of parameters $x_1, \ldots, x_d$ the $k$-subalgebra of $R$ generated by the system of parameters, $k[x_1, \dots, x_n]$,  is isomorphic to a polynomial ring. Since $R$ is complete, it therefore contains the completion $A = k[[x_1, \ldots, x_d]]$ and, since $x_1, \ldots, x_d$ form a system of parameters, $R$ is module-finite over $A$. Of course, there is a lot of freedom in choosing a system of parameters; the Cohen--Gabber Theorem asserts that $x_1,\ldots,x_d$ can be chosen so that $R$ is generically separable over $k[[x_1,\ldots,x_d]].$

\begin{Theorem}[{\cite{GabberOrgogozo}}\footnote{We refer the reader to (\cite{CohenGabber}) for an elementary proof of Theorem~\ref{Cohen-Gabber}.}]\label{Cohen-Gabber} Let $(R,\m,k)$ be a reduced complete local ring of equidimension $d$ and of prime characteristic $p$. There exists a system of parameters $x_1,...,x_d$ of $R$ and coefficient field $k\subseteq R$ such that $R$ is finite and generically separable over the power series ring $k[[x_1,\dots,x_d]]$.
\end{Theorem}

It is useful to keep in mind that that generic separability is equivalent to generic \'etaleness.

\section{Main Theorem and its applications}\label{Main Results}

We say that a sequence $f_1, \ldots, f_c \in \m$ is a parameter sequence if $\dim R/(f_1, \ldots, f_{i + 1}) < \dim R/(f_1, \ldots, f_i)$ for all $i$. We use $\ul{f}$ to denote a sequence of elements $f_1, \ldots, f_c$. If $I$ is an ideal, we say $\ul{f} \in \pow I c$ if each element is in $I$. We add sequences component-wise, so $\ul{f} + \ul{g}$ is the sequence $f_1 + g_1, \cdots, f_c + g_c$.
We will use $\length (M)$ to denote the length of a module, and $\mu (M)$ to denote the minimal number of generators.

We borrow the following property observed in \cite[Lemma~1]{SrinivasTrivedi}.

\begin{Lemma}\label{Basic}
Let $(R,\m,k)$ be a local ring and $\ul{f} \in \pow{\m} c $ be a parameter sequence. Then for any ideal $I$ such that $\length (R/(I, \ul{f}))$ is finite, there exists $N > 0$ such that for all $\ul{\epsilon} \in \pow{\m^N} c$
\[
(I, \ul{f}) = (I, \ul{f} + \ul{\epsilon}).
\]
In particular, $\length (R/(I, \ul{f})) = \length (R/(I, \ul{f}+ \ul{\epsilon})$ for all $\ul{\epsilon} \in \pow{\m^N} c$.
\end{Lemma}
\begin{proof}
Take $N$ such that $\m^N \subseteq (I, \ul{f})$. Then $(I, \ul{f}) = (I, \ul{f} + \ul{\epsilon})$
for all $\ul{\epsilon} \in \pow{\m^N} c$.
\end{proof}

As observed in \cite[Lemma~2]{SrinivasTrivedi}, 
a small enough perturbation of a system of parameters remains  a system of parameters.

\begin{Corollary}\label{stays sop}
Let $(R,\m,k)$ be a local ring. Suppose $\ul{f} \in \pow{\m} c$ is a part of a system of parameters. 
Then there exists $N\in\N$ such that for each $\ul{\epsilon} \in \pow{\m^N} c$, 
$\ul{f}+\ul{\epsilon}$ is a part of a system of parameters.
\end{Corollary}
\begin{proof}
Let $d$ be the dimension of the ring.
We may complete $\ul{f}$ to a full system of parameters $\ul{f}, y_1, \ldots, y_{d-c}$. 
Setting $I = (y_1, \ldots, y_{d-c})$ in the previous lemma we obtain $N$ such that for any $\ul{\epsilon} \in \pow{\m^N} c$
\[
(f_1, \ldots, f_c, y_1, \ldots, y_{d - c}) = (f_1 + \epsilon_1, \ldots, f_c + \epsilon_c, y_1, \ldots, y_{d - c}).
\]
In particular $f_1 + \epsilon_1, \ldots, f_c + \epsilon_c, y_1, \ldots, y_{d - c}$ are $d$ elements 
that generate an $\m$-primary ideal, so they form a system of parameters.
\end{proof}

\begin{Lemma}\label{Dis lemma}
Let $(R, \mathfrak m, k)$ be a complete Cohen-Macaulay local ring with coefficient field $k\subseteq R$ and let 
$f_1, \ldots, f_c, x_1, \ldots x_{d-c}$ be a system of parameters for $R$. Let $A = k[[T_1, \ldots, T_{d-c}]]$ be the power series ring in $d-c$ variables over the field $k$. 
There exists $C \geq 0$ such that for all $n \geq 1$ and all 
$\ul{\epsilon} \in \pow{\mathfrak m^{Cn}} c$, 
\[f_1 + \epsilon_1, \ldots, f_{c} + \epsilon_c, x_{1}, \ldots, x_{d-c}\] 
is a system of parameters and
\[
\Dis_A (R/(\ul{f})) \equiv \Dis_A (R/(\ul{f} + \ul{\epsilon})) \pmod{\mathfrak m_A^n},
\]
where $R/(\ul{f})$ and $R/(\ul{f} + \ul{\epsilon})$ are $A$-algebras by mapping $T_i \mapsto x_i$ and identifying the coefficient field of $A$ with the induced coefficient fields of $R/(\ul{f})$ and $R/(\ul{f}+\ul{\epsilon})$.
\end{Lemma}
\begin{proof}
We can choose $C$ as in Lemma~\ref{Basic}, so for all $\ul{\epsilon} \in \pow {\mathfrak \m^C} c$ we have
$(\ul{f} + \ul{\epsilon}, x_1, \ldots, x_{d-c}) = (\ul{f}, x_1, \ldots, x_{d-c})$ and thus $\ul{f}+\ul{\epsilon}, x_1, \ldots, x_{d-c}$ is a system of parameters. 
Hence $R/(\ul{f})$ and $R/(\ul{f} + \ul{\epsilon})$ are free $A$-modules of the same rank
$\length_A (R/(\ul{f}, x_1, \ldots, x_{d-c}))$.  
Moreover, if $r_1,\ldots,r_n\in R$ are chosen so that their images in $R/(\ul{f})$ serve as a basis for $R/(\ul{f})$ as an $A$-module, then for each $\ul{\epsilon}\in \pow {\m^C} c$ then,
because the images of $r_1, \ldots, r_n$ form a basis of 
$R/(\ul{f} + \ul{\epsilon}, \m_A) = R/(\ul{f}, x_1, \ldots, x_{d-c})$,
the images of $r_1,\ldots,r_n$ in $ R/(\ul{f}+\ul{\epsilon})$ are a basis of $R/(\ul{f} + \ul{\epsilon})$ too.

Observe that for each $n\geq 1$, $\mathfrak m^{Cn} \subseteq (x_1, \ldots, x_{d-c})^n + (\ul{f})$. 
Thus for each $\ul{\epsilon} \in \pow{\m^{Cn}} c$, 
$R/(\ul{f}, \mathfrak m_A^n) = R/(\ul{f} + \ul{\epsilon}, \mathfrak m_A^n)$ as $A/\mathfrak m_A^n$-algebras.
Then by part (\ref{Discriminant Properties Part 2}) of Lemma~\ref{Discriminant Properties}
\begin{align*}
\Dis_A (R/(\ul{f})) \mod \mathfrak m_A^nR/(\ul{f}) &= \Dis_{A/\mathfrak m_A^n} (R/(\ul{f}, \mathfrak m_A^n))\\
&= \Dis_{A/\mathfrak m_A^n} (R/(\ul{f}+\ul{\epsilon}, \mathfrak m_A^n))\\
& = \Dis_A (R/(\ul{f} + \ul{\epsilon})) \mod \mathfrak m_A^nR/(\ul{f} + \ul{\epsilon}),
\end{align*}
where the middle equality holds since we have a common basis $r_1,\ldots,r_n\in R$ 
for  $R/(\ul{f}+\ul{\epsilon})$ and $R/(\ul{f})$ as $A$-modules.
\end{proof}

Suppose $(R,\m,k)$ satisfies the hypotheses of Lemma~\ref{Dis lemma}. 
The following corollary shows that if $f\in\m$ is a parameter element and $R/(f)$ is generically separable over a regular ring $A$, then $R/(g)$ is generically separable over the same regular local ring $A$, provided $g$ is a small enough perturbation of $f$.

\begin{Corollary}\label{separability cor}
Let $(R, \mathfrak m, k)$ be a complete Cohen-Macaulay local ring with coefficient field $k\subseteq R$. 
Let $f_1, \ldots, f_c, x_1, \ldots, x_{d-c}$ be a system of parameters for $R$. Let $A = k[[T_1, \ldots, T_{d-c}]]$ be the power series ring in $d-c$ variables over the field $k$ and  
consider  $R/(\ul{f})$ is an $A$-algebra by mapping $T_i\mapsto x_i$ and identifying the coefficient field of $A$ with the induced coefficient field of $R/(\ul{f})$. Suppose that $A \subseteq R/(\ul{f})$ is generically separable. Then there exists $N\geq 1$ such that for any 
$\ul{\epsilon} \in \pow{\mathfrak m^N} c$, $A \subseteq R/(\ul{f} + \ul{\epsilon})$ is generically separable 
as an $A$-algebra given by mapping $T_i\mapsto x_i$ and identifying the coefficient field of $A$ with the induced coefficient field of $R/(\ul{f} + \ul{\epsilon})$.
\end{Corollary}
\begin{proof}
The assumption $ R/(\ul{f})$ is generically separable over $A$ is equivalent to $\Dis_A(R/(\ul{f}))\not=0$ by Part~(\ref{Discriminant Properties Part 1}) of Lemma~\ref{Discriminant Properties}.  Let $C\geq 0$ as given by Lemma~\ref{Dis lemma}. By Krull's intersection theorem we can choose $n\geq 1$ so that $\Dis_A(R/(\ul{f}))\not\equiv 0\mod \m^n_A$. Then for any $\ul{\epsilon} \in \pow{\m^{Cn}} c$, $\Dis_A(R/(\ul{f}+\ul{\epsilon}))\equiv\Dis_A(R/(\ul{f}))\mod \m^n_A\not=0$,  so $\Dis_A(R/(\ul{f}+\ul{\epsilon}))\not=0$. \end{proof}

\subsection{Continuity of Hilbert--Kunz multiplicity}

It is time to set up the uniform convergence machinery. We closely follow \cite[Lemma~2.3, Theorem~3.2]{PolstraTucker}.

Let $(R, \m, k)$ be an F-finite local ring and $I$ be an $\m$-primary ideal. 
Then 
$F_*^e R \otimes_R R/I \cong F_*^e R/I^{[p^e]}$ and we may compute
\[
\length_R (F_* R \otimes_R R/I) = [k : k^{p}] \length_R (R/I^{[p]}).
\]

\begin{Theorem}\label{main eHK thm}
Let $(R, \mathfrak m,k)$ be an F-finite Cohen-Macaulay local ring of prime characteristic $p$ and of dimension $d+c$. 
If $\ul{f} = f_1, \ldots, f_c$ is a parameter sequence such that $\hat{R}/(\ul{f})\hat{R}$ is reduced,  then there exists a constant $C \geq 0$ and an integer $N\geq 1$ such that for any $e\geq 1$ and  $\ul{\epsilon} \in \pow {\mathfrak m^N}{c}$
\[
\left |\length \left (R/(\ul{f} + \ul{\epsilon}, \mathfrak m^{[p^{e+1}]}) \right ) - 
p^{d} \length \left (R/(\ul{f} + \ul{\epsilon}, \m^{[p^e]}) \right) \right| \leq C p^{e(d-1)}.
\]
\end{Theorem}

\begin{proof}
Without loss of generality, we may assume $R$ is complete. Because $R$ is Cohen-Macaulay, $R/(\ul{f})$ is equidimensional and reduced. By Theorem~\ref{Cohen-Gabber} there exists coefficient field $k\subseteq R/(\ul{f})$ and parameters $x_1,...,x_d\in R$ such that $x_1,...,x_d$ is a system of parameters for $R/(\ul{f})$ and $R/(\ul{f})$ is module-finite and generically separable over the regular local ring $A := k[[x_1,\ldots,x_d]]$.

We observe that the coefficient field can be lifted from $R/(\ul{f})$ to $R$. 
Essentially, the proof of \cite[Theorem~28.3]{MatsumuraNew}
shows that a coefficient field is determined by the lifts of a $p$-basis of $k$ 
(see also \cite[Theorem, page 12]{HochsterNotes}).
Thus we may just lift the $p$-basis to $R$.

By Lemma~6.5 of \cite{HHJAMS} (the proof still holds if we replace $R^\infty$ by $R^{1/p}$), 
we have an exact sequence
\[
0 \to R/(\ul{f})[A^{1/p}] \to (R/(\ul{f}))^{1/p} \to M \to 0
\]
where $0 \neq c = \Dis_A (R/(\ul{f}))$ annihilates $M$.

By Lemma~\ref{Dis lemma} and Corollary~\ref{separability cor},
we can find $N$ such that $A \subseteq R/(\ul{f} + \ul{\epsilon})$ is still generically separable.
Therefore we have an exact sequence
\[
0 \to R/(\ul{f} + \ul{\epsilon})[A^{1/p}] \to (R/(\ul{f} + \ul{\epsilon}))^{1/p} \to M_{\ul{\epsilon}} \to 0
\]
where $c_{\ul{\epsilon}} = \Dis_A (R/(\ul{f} + \ul{\epsilon}))$ annihilates $M_{\ul{\epsilon}}$ and 
$c_{\ul{\epsilon}} - c \in \mathfrak m_A^N$.

By \cite{Kunz1969} $A^{1/p}$ is a free $A$-module of rank $[k: k^p] p^d$, so 
tensoring the above sequence with $\m^{[p^e]}$
we get that
\[
[k: k^p] \length \left(R/(\ul{f} + \ul{\epsilon}, \m^{[p^{e + 1}]}) \right) - [k: k^p]p^{d} 
\length \left( R/(\ul{f} + \ul{\epsilon}, \m^{[p^e]}) \right) \leq 
\length \left(M_{\ul{\epsilon}}/\m^{[p^e]}M_{\ul{\epsilon}} \right).
\]
Since $M_{\ul {\epsilon}}$ is a $R/(\ul{f} + \ul{\epsilon},  c_{\ul{\epsilon}})$-module and 
a quotient of $(R/(\ul{f} + \ul{\epsilon}))^{1/p}$, 
\[
\length \left(M_{\ul{\epsilon}}/\m^{[p^e]}M_{\ul{\epsilon}} \right) 
\leq \length \left (R/(\ul{f} + \ul{\epsilon}, c_{\ul{\epsilon}}, \m^{[p^e]}) \right) \mu \left ((R/(\ul{f} + \ul{\epsilon}))^{1/p} \right).
\]
Furthermore, it easily follows from the inclusion $(x_1, \ldots, x_d) \subseteq \mathfrak m$ that
\[
\mu \left ((R/(\ul{f} + \ul{\epsilon}))^{1/p} \right) \leq \length \left( (R/(x_1^p, \ldots, x_d^p, \ul{f} + \ul{\epsilon}) )^{1/p} \right),
\] 
thus
\[
\mu \left ((R/(\ul{f} + \ul{\epsilon}))^{1/p} \right) \leq 
[k: k^p] \length \left( R/(x_1^p, \ldots, x_d^p, \ul{f} + \ul{\epsilon})  \right) \leq [k:k^p]p^d \length \left (R/(x_1, \ldots, x_d, \ul{f}) \right),
\]
where the last inequality follows by filtration.

Now, we can complete $0 \neq c$ to a system of parameters $c, t_1, \ldots, t_{d-1}$ in $A$ and,
after increasing $N$ if necessarily, by Lemma~\ref{Dis lemma} we can assume that 
$c_{\ul{\epsilon}}, t_1, \ldots, t_{d-1}$ is still a system of parameters and that
$(\ul{f}+\ul{\epsilon}, c_{\ul{\epsilon}}, t_1, \ldots, t_{d-1}) = (\ul{f}, c, t_1, \ldots, t_{d-1})$.
Therefore
\[
\length \left (R/(\ul{f} + \ul{\epsilon}, c_{\ul{\epsilon}}, \mathfrak m^{[p^e]}) \right) \leq 
\length \left (R/ (\ul{f}+\ul{\epsilon}, c_{\ul{\epsilon}}, t_1^{p^e}, \ldots, t_{d-1}^{p^e}) \right)
\leq p^{e(d-1)} \length \left (R/ (\ul{f}, c, t_1, \ldots, t_{d-1}) \right).
\]
Let
\[
C = p^d \length (R/ (\ul{f}, c, t_1, \ldots, t_{d-1})) \length (R/(x_1, \ldots, x_d, \ul{f}))
\]
and observe that we found one of the two required inequalities:
\[
\length \left (R/(\ul{f} + \ul{\epsilon}, \mathfrak m^{[p^{e+1}]}) \right ) - 
p^{d} \length \left (R/(\ul{f} + \ul{\epsilon}, \m^{[p^e]}) \right)  \leq C p^{e(d-1)}.
\] 

For the other inequality, we can repeat the proof for the exact sequence
\[
0 \to (R/(\ul{f}))^{1/p} \xrightarrow{c} R/(\ul{f})[A^{1/p}]  \to L \to 0
\]
where the map is multiplication by $c = \Dis_A (R/(\ul{f}))$.
Again the cokernel is annihilated by $c$, because $R/(\ul{f})[A^{1/p}] \subseteq (R/(\ul{f}))^{1/p}$. Moreover, we get for each $\ul{\epsilon}$ a short exact sequence
\[
0 \to (R/(\ul{f}+\ul{\epsilon}))^{1/p} \xrightarrow{c_{\ul{\epsilon}}} R/(\ul{f}+\ul{\epsilon})[A^{1/p}]  \to L_{\epsilon} \to 0.
\]
Most of the proof carries through as above, except that now we estimate
\[
\length \left (L_{\ul{\epsilon}}/\m^{[p^e]}L_{\ul{\epsilon}} \right)  
\leq  \mu \left ( R/(\ul{f})[A^{1/p}] \right) \length \left (R/(\ul{f} + \ul{\epsilon}, c_{\ul{\epsilon}}, \m^{[p^e]}) \right)
= [k:k^p]p^d \length \left (R/(\ul{f} + \ul{\epsilon}, c_{\ul{\epsilon}}, \m^{[p^e]}) \right)
\]
and from there obtain 
\[
p^{d} \length \left (R/(\ul{f} + \ul{\epsilon}, \m^{[p^e]}) \right)
- \length \left (R/(\ul{f} + \ul{\epsilon}, \mathfrak m^{[p^{e+1}]}) \right )  
\leq  \left( p^d[k : k^p] \length (R/ (\ul{f}, c, t_1, \ldots, t_{d-1})) \right)p^{e(d-1)}.
\]
\end{proof}

Following the treatment in \cite{TuckerF-sigExists}, we obtain the following corollary:

\begin{Corollary}\label{Corollary to main eHK thm}
Let $(R, \mathfrak m,k)$ be an F-finite Cohen-Macaulay local ring of prime characteristic $p$ and of dimension $d+c$. If $\ul{f}$ is a parameter sequence of length $c$ and such that $\hat{R}/(\ul{f})\hat{R}$ is reduced, then there exists a constant $C \geq 0$ and an integer $N\geq 1$ such that for any $e\geq 1$ and $\ul{\epsilon} \in \pow{\mathfrak m^N} c$
\[
\left |\ehk (R/(\ul{f}+\ul{\epsilon})) - p^{-ed} \length \left (R/(\ul{f} + \ul{\epsilon}, \m^{[p^e]}) \right) \right| \leq C p^{-e}.
\]
\end{Corollary}
\begin{proof}
Using induction on $q'$ as in the proof of \cite[Proposition~3.4]{TuckerF-sigExists} 
we may obtain from Theorem~\ref{main eHK thm} that
\[
\left |\length \left (R/(\ul{f} + \ul{\epsilon}, \mathfrak m^{[p^{e+e'}]}) \right ) - 
p^{e'd} \length \left (R/(\ul{f} + \ul{\epsilon}, \m^{[p^e]}) \right) \right| 
\leq C p^{(e + e' - 1)(d-1)}(1 + p + \cdots + p^{e' - 1}).
\]
Thus, dividing by $p^{(e + e')d}$ we get that
\[
\left | p^{-(e + e')d}\length \left (R/(\ul{f} + \ul{\epsilon}, \mathfrak m^{[p^{e+e'}]}) \right ) - 
p^{-ed} \length \left (R/(\ul{f} + \ul{\epsilon}, \m^{[p^e]}) \right) \right|
\leq C \frac{(1 + p + \cdots + p^{e' - 1})}{p^{d - 1} p^{e + e'}}.
\]
Note that
\[
C \frac{(1 + p + \cdots + p^{e' - 1})}{p^{d - 1} p^{e + e'}}
= \frac{C(p^{e'} - 1)}{(p - 1) p^{d - 1} p^{e'}} p^{-e} \leq C p^{-e},
\]
so the claim follows after letting $e' \to \infty$.
\end{proof}

Continuity of Hilbert--Kunz multiplicity follows from Corollary~\ref{Corollary to main eHK thm}.

\begin{Corollary}\label{HK continuity}
Let $(R, \mathfrak m, k)$ be a Cohen-Macaulay F-finite local ring of prime characteristic $p$ and dimension $d+c$. If $\ul{f}$ is a parameter sequence of length $c$ and  such that $\hat{R}/(\ul{f})\hat{R}$ is reduced, then for any $\delta > 0$ there exists an integer $N > 0$ such that for any $\ul{\epsilon} \in \pow{\mathfrak m^N} c$
\[
\left |\ehk (R/(\ul{f})) - \ehk(R/(\ul{f} + \ul{\epsilon})) \right| < \delta.
\]
\end{Corollary}
\begin{proof}
By Lemma~\ref{Basic}, for any given $e$ there exists $N$ such that 
\[
\length \left (R/(\ul{f} + \ul{\epsilon}, \m^{[p^e]}) \right) =
\length \left (R/(\ul{f}, \m^{[p^e]}) \right)
\]
for all $\ul{\epsilon} \in \pow{\m^N}c$.
By Corollary~\ref{Corollary to main eHK thm}, we may further assume that for some constant $C > 0$
\[
\left |\ehk (R/(\ul{f} + \ul{\epsilon})) - \frac{1}{p^{ed}}
\length \left (R/(\ul{f} + \ul{\epsilon}, \m^{[p^e]}) \right) \right| < Cp^{-e}.
\]
Thus
\[
|\ehk (R/(\ul{f} + \ul{\epsilon})) - \ehk(R/(\ul{f})) | < 2Cp^{-e}
\]
and the claim follows.
\end{proof}

\subsection{Continuity of F-signature}

To establish continuity of F-signature we will need an effective way of comparing the $e$th Frobenius splitting numbers of rings which are the homomorphic image of a common ring. We begin by recalling the following theorem of Huneke and Leuschke.

\begin{Theorem}{\cite[Proof of Theorem~11]{HunekeLeuschke}}
\label{Gorenstein splitting numbers}
Let $(R,\m,k)$ be an $F$-finite local Gorenstein ring of prime characteristic $p>0$ and Krull dimension $d$. Suppose $x_1,\ldots,x_d$ is a system of parameters of $R$ and $u$ generates the socle mod $(x_1,\ldots,x_d)$. Then for each $e\in\mathbb{N}$
\[
a_e(R)=[k:k^{p^e}]\lambda(R/((x_1,\ldots,x_d)^{[p^e]}:u^{p^e})).
\]
\end{Theorem}

\begin{Corollary}\label{splitting number}
 Let $(R,\m,k)$ be a local Gorenstein F-finite ring of prime characteristic $p$. Let $\ul{f}$ be a regular sequence of length $c$. Then for any integer $e$ there exists an integer $N$ such that $a_e(R/(\ul{f})) = a_e(R/(\ul{f} + \ul{\epsilon}))$ for any $\ul{\epsilon} \in \pow {\m^N} c$. 
\end{Corollary}

\begin{proof} If $\dim(R) = c$, then $R/(f)$ is artinian and, thus, has a splitting if an only if it is a field.
If we take $\ul{\epsilon} \in \pow{\m^2} c$, then $(\ul {f}) = \m$ if and only if $(\ul{f}+ \ul{\epsilon}) = \m$.
So the Frobenius splitting numbers of $R/(\ul{f})$ and $R/(\ul{f}+\ul{\epsilon})$ will be the same, either $0$ or $[k : k^{p^e}]$, the latter of which occurs if and only if $(\ul{f}) = \m$.

Suppose $\dim(R)=d+c > c$ and let $x_1,\ldots, x_d$ be a choice of system of parameters for $R/(\underline{f})$ and $u\in R$ generates the socle mod $(\underline{f},x_1,\ldots, x_d)$. By the proof of Corollary~\ref{stays sop} there exists an integer $N$ so that for each $\underline{\epsilon}\in (\m^{N})^{\oplus c}$ one has that $(\underline{f}+\underline{\epsilon}, x_1,\ldots,x_d)=(\underline{f},x_1,\ldots,x_d)$ and $(\underline{f}+\underline{\epsilon}, x^{p^e}_1,\ldots,x^{p^e}_d)=(\underline{f},x^{p^e}_1,\ldots,x^{p^e}_d)$. For such choices of $\underline{\epsilon}$ the sequence $x_1,\ldots,x_d$ is a full system of parameters for $R/(\underline{f}+\underline{\epsilon})$, $u$ generates the socle of $R/(\underline{f}+\underline{\epsilon},x_1,\ldots,x_d)$, and by multiple applications of Theorem~\ref{Gorenstein splitting numbers}
\begin{align*}
a_e(R/(\underline{f}))&=[k:k^{p^e}]\lambda(R/((\underline{f},x_1^{p^e},\ldots,x_d^{p^e}):u^{p^e}))\\
&=[k:k^{p^e}]\lambda(R/((\underline{f}+\underline{\epsilon},x_1^{p^e},\ldots,x_d^{p^e}):u^{p^e}))=a_e(R/(\underline{f}+\underline{\epsilon}).
\end{align*}
\end{proof}

Less assumptions are needed for an inequality. 

\begin{Proposition}\label{Splitting numbers local maximum proposition} 
Let $(R, \m,k)$ be a complete F-finite ring of positive characteristic $p > 0$. 
Then for any parameter sequence $\ul{f}$ of length $c$ and a fixed positive integer $e$ there exists integer $N$ such that for all $\ul{\epsilon} \in \pow{\m^N} c$
\[
a_e (R/(\ul{f}+\ul{\epsilon})) \leq a_e (R/(\ul{f})).
\]
\end{Proposition}
\begin{proof}
Since $R$ is complete, we can represent it as a quotient of a regular local ring $(S, \m)$, $R = S/I$.
Then by \cite[Proposition~3.1]{EnescuYao}
\[
a_e(R/(f)) = [k : k^{p^e}]\length_S \left( \frac{(I, \ul{f})^{[p^e]}:_S (I,\ul{f})+\m^{[p^e]}}{\m^{[p^e]}}\right).
\]
Let $\ul{\epsilon} \in \pow{\m^N} c$ for an arbitrary $N$. Then
\[
(I, \ul{f} + \ul{\epsilon})^{[p^e]}:_S (I,\ul{f}+\ul{\epsilon}) \subseteq 
\left( (I, \ul{f})^{[p^e]} + \m^N \right ) :_S (I, \ul{f} + \ul{\epsilon}). 
\]
Since $\epsilon_i$ are in $\m^N$, 
\[
\left( (I, \ul{f})^{[p^e]} + \m^N \right ) :_S (I, \ul{f} + \ul{\epsilon}) = \left( (I, \ul{f})^{[p^e]} + \m^N \right ) :_S (I, \ul{f}).
\]
By Krull's intersection theorem
\[
\bigcap_{n = 1}^{\infty}  \left( (I, \ul{f})^{[p^e]} + \m^n \right ) :_S (I, \ul{f}) = \left( (I, \ul{f})^{[p^e]}\right ) :_S (I, \ul{f}),
\]
so by Chevalley's lemma we could have chosen $N$ so that 
\[
\left( (I, \ul{f})^{[p^e]} + \m^N \right ) :_S (I, \ul{f}) \subseteq \left( (I, \ul{f})^{[p^e]}\right ) :_S (I, \ul{f}) + \m^{[p^e]}.
\]
In which case $a_e(R/(\ul{f} + \ul{\epsilon})) \leq a_e(R/(\ul{f}))$. 
\end{proof}

\begin{Theorem}\label{not main thm}
Let $(R, \mathfrak m, k)$ be a Gorenstein F-finite ring of prime characteristic $p$ and dimension $d+c$. If $\ul{f}$ is a parameter sequence of length $c$ such that $\hat{R}/(\ul{f})\hat{R}$ is reduced then for any $\delta > 0$ there exists an integer $N > 0$ such that for any $\ul{\epsilon} \in \pow{\mathfrak m^N} c$
\[
|\s (R/(\ul{f})) - \s (R/(\ul{f} + \ul{\epsilon})) | < \delta.
\]
\end{Theorem}
\begin{proof}
We may assume $R$ is complete. By Theorem~\ref{Cohen-Gabber} we may choose parameters $x_1,\ldots, x_d\in R$ such that $x_1,\ldots, x_d$ is a system of parameters for $R/(\ul{f})$ and $R/(\ul{f})$ is module finite and generically separable over the regular local ring $A:=k[[x_1,\ldots,x_d]]$. In which case we have short exact sequence
\[
0 \to R/(\ul{f})[A^{1/p}] \to (R/(\ul{f}))^{1/p} \to M \to 0
\]
and $0 \neq c = \Dis_A (R/(\ul{f}))$ annihilates $M$.

Let $M$ and $M_{\ul{\epsilon}}$ be as in the proof of Theorem~\ref{main eHK thm}. In which case,  there are isomorphisms $R/(\ul{f} + \ul{\epsilon})[A^{1/p}] \cong \oplus^{p^d [k : k^p]} (R/(\ul{f} + \ul{\epsilon}))$ and short exact sequences
\[
0 \to R/(\ul{f} + \ul{\epsilon})[A^{1/p}] \to (R/(\ul{f} + \ul{\epsilon}))^{1/p} \to M_{\ul{\epsilon}} \to 0.
\]

Apply the exact functor $(-)^{1/p^e}$ to the above to get exact sequence
\[
0 \to \bigoplus^{p^d [k : k^p]}\left (  (R/(\ul{f} + \ul{\epsilon})^{1/p^e}) \right) \to 
 (R/(\ul{f} + \ul{\epsilon}))^{1/p^{e+1}} \to  M_{\ul{\epsilon}}^{1/p^e} \to 0.
\]
By (\cite[Lemma~2.1]{PolstraTucker})
\[
\frk \left (  (R/(\ul{f} + \ul{\epsilon}))^{1/p^{e+1}}  \right ) \leq p^d[k : k^p] 
\frk \left (  (R/(\ul{f} + \ul{\epsilon}))^{1/p^e}  \right) + 
\mu ((M_{\ul{\epsilon}})^{1/p^e}),
\]
or, equivalently,
\[
a_{e+1}(R/(\ul{f}+\ul{\epsilon}))\leq p^d[k : k^p]a_e(R/(\ul{f}+\ul{\epsilon}))+\mu ((M_{\ul{\epsilon}})^{1/p^e}).
\]
Techniques used in the proof of Theorem~\ref{main eHK thm} give the existence of a constant $C$ such that for all $\ul{\epsilon}\in \pow{\mathfrak m^N} c$ 
\[
\mu ((N_{\ul{\epsilon}})^{1/p^e}) = \length (N_{\ul{\epsilon}}/\m^{[p^e]}N_{\ul{\epsilon}}) 
\leq Cp^{e(d-1)}.
\]

Let $L$ and $L_{\ul{\epsilon}}$ be as in the proof Theorem~\ref{main eHK thm}. Similarly, we can also bound
\[ 
p^d[k : k^p] \frk \left ( (R/(\ul{f} + \ul{\epsilon}))^{1/p^e} \right) -  \frk \left ( (R/(\ul{f} + \ul{\epsilon}))^{1/p^{e + 1}}  \right ) 
\leq \length (L_{\ul{\epsilon}}/\m^{[p^e]}L_{\ul{\epsilon}})
\]
and can once again obtain constant $C$ independent of $\ul{\epsilon}$ so that
\[|a_{e+ 1} (R/(\ul{f} + \ul{\epsilon})) - p^d[k : k^p] a_e (R/(\ul{f} + \ul{\epsilon}))| < Cp^{e(d-1)}.\]
Since $\rank R^{1/p^e} = p^{ed} [k : k^{p^e}]$ (\cite[Proposition~2.3]{Kunz1976}), 
as explained in Corollary~\ref{HK continuity}, this gives us a uniform convergence statement:
there exist $D, N > 0$ such that for all $\ul{\epsilon} \in \pow{\m^N} c$ 
\[
\left |\s (R/(\ul{f}+ \ul{\epsilon})) - \frac 1{\rank R^{1/p^e}} a_e (R/(\ul{f} + \ul{\epsilon})) \right| < \frac{D}{p^e}.
\]
The statement now follows by employing Corollary~\ref{splitting number} and using the methods of the proof of Corollary~\ref{HK continuity}.
\end{proof}

\section{Examples and questions}

In this section we want to present examples relevant to the results in Section~\ref{Main Results}. The first example shows that we cannot expect Hilbert--Kunz multiplicity to be locally constant at reduced parameter elements.

\begin{Example}\label{not constant example}
Consider an element $f = xy$ in $R = k[[x,y,t]]$. An easy computation shows that 
$\ehk (R/ (xy)) = 2$, while by \cite[Theorem~3.1]{Conca} $\ehk (R/(xy + t^n)) = 2 - 1/n$ for $n \geq 2$.  
Thus 
\begin{enumerate}
\item there is no neighborhood of $xy$ such that $\ehk$ is constant.
\item Of course, $\eh (R/(xy + t^n)) = 2 = \eh (R/(xy))$ for all $n \geq 2$.
\item Moreover,  we have $\Gr_{(x,y,t)} (R/(xy + t^n)) = k[x,y,t]/(xy) = \Gr_{(x,y,t)} (R/(xy))$ for all $n \geq 3$. 
\end{enumerate}
\end{Example} 

\begin{Example}\label{fsig}
Essentially the same example works for F-signature. Clearly, $k[[x,y,t]]/(xy)$ is not a domain, so the F-signature is $0$. 
On the other hand, the F-signature of $k[[x,y,t]]/ (xy + t^{n})$ is known to be $1/n$
(\cite[Example~18]{HunekeLeuschke}). 

\end{Example}

By comparison, it is not that hard to see that Hilbert--Samuel multiplicity is locally constant at parameter elements in regular local ring. In fact, Srinivas and Trivedi (\cite[Corollary~5]{SrinivasTrivedi}) 
showed that the entire Hilbert function is locally constant.

\subsection{Cohen-Macaulayness}

The Cohen-Macaulay assumption of our main results may not be completely necessary. For a possible extension, one could recall that Srinivas and Trivedi (\cite[Corollary~5]{SrinivasTrivedi})
show continuity of the Hilbert--Samuel multiplicity if $R$ is generalized Cohen-Macaulay, i.e., 
$\length (\lc_\m^i (R) ) < \infty$ for all $i < \dim R$.

We also want to give two examples that show that one cannot relax Cohen-Macaulayness too much.
The first example comes from \cite[Example~1]{SrinivasTrivedi}.

\begin{Example}
Let $R$ be the localization at $\m = (x,y,z)$ of $k[x,y,z]/(xy, xz)$. 
Note that $R$ is not equidimensional.
Because $\dim R/(y) = 1$, the two multiplicities coincide and we easily see that
\[\ehk (R/(y)) = \eh (R/(y)) = \eh (k[x,z]_\m/(xz)) = 2.\]

On the other hand, $R/(y + x^N) = k[x,z]_\m/(x^{N+1}, xz)$ has the only minimal prime $(x)$ 
for all $N$. Hence, by the associativity formula,
\[
\ehk (R/(y + x^N)) = \eh(R/(y + x^N)) = \eh (R/(x)) \length (R_{(x)}) = 1.
\]
\end{Example}

\begin{Proposition}\label {bad LC}
Let $R$ be the subring 
$k[[x^3, x^2y, y^3, y^2z, z^3, z^2x]] \subseteq k[[x,y]]$. 
Observe that $x^3, y^3, z^3$ form a system of parameters and consider a family of ideals $J_k = (x^3, y^3 + z^{3k})$. 
Then:
\begin{enumerate}
\item $\ehk (R/(x^3, y^3)) = \eh (R/(x^3, y^3)) = 11$,
\item $\ehk (R/J_k) = \eh (R/J_k) \neq 11$ for all $k \neq 1 \pmod 3$. 
\end{enumerate}
\end{Proposition}
\begin{proof}
First, we observe that $\dim R/(x^3, y^3) = \dim R/J_k = 1$, so multiplicity equals to Hilbert--Kunz multiplicity.

We compute $\eh (R/(x^3, y^3))$ by using the associativity formula. 
Observe that 
\[
P = (x^3, x^2y, y^3, y^2z, z^2x) = \sqrt{(x^3, y^3)}
\] 
is prime. 
Thus the associativity formula gives that $\ehk (R/(x^3, y^3)) = \ehk (R/P) \length (R_P/(x^3, y^3))$.

The quotient $R/P \cong k[z^3]$ is regular, so $\ehk (R/P) = 1$.
Now, 
\[
\widehat {R_P}/(x^3, y^3)  \cong k(z^3)[[x^3, y^3, x^2y, y^2/z^2, x/z]]/(x^3, y^3)
\]
and it is easy to compute the basis of this ring over $k(z^3)$:
\[
1, x/z, y^2/z^2, x^2/z^2, x^2y, xy^2, x^2y^2/z, x^3y/z, y^4/z, xy^4/z^2, x^4y/z^2. 
\]
Therefore, $\eh (R/(x^3, y^3)) = 11$.

The computation for $J_k$ is similar but a bit more complicated. 
First of all, if $k \neq 1 \pmod 3$ (so $2k+1$ is not divisible by $3$), then 
\[
Q = (x^3, x^2y, z^2x, y^3 - z^{3k}) = \sqrt{J_k}
\]
is prime as can be seen by taking the presentation
\[
k[[a, b]]/(a^3 - b^{2k + 1}) \to R/Q \cong k[[y^2z, z^3]]/(y^3 - z^{3k}).
\]
In this case, $\eh (R/Q) = 3$, so $\eh (R/J_k)$ is divisible by $3$ and we already know that it cannot be equal to $11$.
\end{proof}

Thus even Hibert--Samuel multiplicity does not behave well in this example. 
Observe that $xy^2 = x^3y^3/x^2y$ 
and $xyz = y^3z^2x/y^2z$, so the normalization of $R$ is the Veronese subring
\[
V = k[[x^3, x^2y, xy^2, y^3, y^2z, yz^2, z^3, z^2x, zx^2]].
\] 
Because $xy^2 \cdot y^{3n} \notin R$ for all $n$, the quotient $S/R$ does not have finite length. 
On the other hand, $x^3y^3V \subseteq R$, so $\dim S/R$ must equal to $1$.

Therefore since $\depth V = 2$, from the exact sequence of local cohomology we have
\[
0 \to \lc_\m^0 (S/R) \to \lc_\m^1(R) \to 0 \to \lc_\m^1 (S/R) \to \lc_\m^2 (R) \to 0 = \lc_\m^2(S).
\]
Therefore, $\length(\lc_\m^1(R)) < \infty$, but $\length(\lc_\m^2 (R)) = \infty$.
This example shows that the assumptions of Srinivas and Trivedi are close to being necessary.

\subsection{Questions} The techniques used to prove the main theorems of the paper, which are summarized in Theorem~\ref{Intro theorem}, relies on the assumption that the parameter sequence $f_1,\ldots,f_h$ forms a reduced ideal.

\begin{Question}
Let $(R,\m,k)$ be a Cohen-Macaulay ring of prime characteristic $p$. If $(f_1,\ldots, f_c)$ is a parameter ideal, but $R/(f_1,\ldots, f_c)$ is not necessarily reduced, then given $\delta>0$ does there exits an integer $N>0$ such that for any $\epsilon_1, \ldots ,\epsilon_h\in \m^N$
\begin{enumerate}
\item $|\ehk(R/(f_1,\ldots ,f_c))-\ehk(R/(f_1+\epsilon_1, \ldots, f_c+\epsilon_c))|<\delta$, and
\item (if $R$ is Gorenstein) $|\s(R/(f_1,\ldots ,f_c))-\s(R/(f_1+\epsilon_1, \ldots, f_c+\epsilon_c))|<\delta$?
\end{enumerate}

\end{Question}

Originally, the second author hoped to approach Question~\ref{hypersurface q}, 
by passing to the associated graded ring. Namely, Srinivas and Trivedi (\cite[Theorem~3]{SrinivasTrivedi}) 
show $\Gr_\mathfrak m (R/(f))$ is a locally constant function of $f$ in $\mathfrak m$-adic topology,
which immediately implies that Hilbert-Samuel multiplicity is locally constant.
 
However the Hilbert--Kunz multiplicities of a local ring and its associated graded ring need not 
be equal (\cite[p. 302]{WatanabeYoshida2000} and also Example~\ref{not constant example}).
Moreover, Hilbert--Kunz multiplicity is not an invariant of the associated graded ring, 
i.e., as we see in Example~\ref{not constant example}, it may happen that $\ehk (R) \neq \ehk (S)$ even though 
$\Gr_{\mathfrak m_R}(R) \cong \Gr_{\mathfrak m_S} (S)$. 
Investigation of the following question could yield an explanation.

\begin{Question}\label{convergence q}
Does the convergence rate of the Hilbert--Kunz multiplicity come from the associated graded ring?
Namely, is there a constant $C$ that depends only on the associated graded ring of $S$ such that
\[
\left |\ehk (S) - \frac{1}{p^{ed}} \length \left (S/(\mathfrak m)^{[p^e]} \right) \right| < \frac{C}{p^e}
\]
for all $e$?
\end{Question}

This question can be stated even more explicitly: consider the family of ideals in $\Gr_{\mathfrak m}(R)$ 
\[
J_e = \bigoplus_{n \geq 0} \frac{\mathfrak m^{[p^e]} \cap \mathfrak m^n + \mathfrak m^{n+1}}{\mathfrak m^{n+1}}.
\]
By the definition, $\length (R/\mathfrak m^{[p^e]}) = \length (\Gr_{\mathfrak m} (R))/J_e)$. Because $J_e^{[p]} \subseteq J_{e + 1}$, we can show existence of $\ehk(R)$ directly in $\Gr_{\mathfrak m} (R)$ (\cite[Theorem~4.3]{PolstraTucker}). However, in order to get the convergence rate, the current techniques (\cite[Corollary~4.5]{PolstraTucker}) require us to find a $p^{-1}$-linear map
$\psi$ on $\Gr_{\mathfrak m} (R)$ such that $\psi (J_{e + 1}) \subseteq J_e$ for all $e$.

Because for a fixed $e$ the individual Hilbert--Kunz function 
$R/(f, \mathfrak m^{[p^e]})$ is clearly continuous (Lemma~\ref{Basic}), we 
can employ  the uniform convergence machinery to show that a positive answer to Question~\ref{convergence q} gives a positive answer to Question~\ref{hypersurface q}. However, we currently do not see how to approach Question~\ref{convergence q}, 
so we had to search for other methods. 

We also want to offer two questions motivated by intuitive understanding of singularities: 
we expect that singularity will not get worse after a sufficiently small perturbation.

\begin{Question}\label{worse sing}
Let $(R,\m,k)$ be a local ring of prime characteristic $p$. Does there $N$ such that for all $\epsilon \in \m^N$
$\ehk (R/(f)) \geq \ehk (R/(f + \epsilon))$?
\end{Question}

\begin{Question}\label{worse sing 2}
Let $(R,\m,k)$ be a local Gorenstein ring of prime characteristic $p$. Does there $N$ such that for all $\epsilon \in \m^N$
$\s (R/(f)) \leq \s (R/(f+\epsilon))$?
\end{Question}

We remark that Srinivas and Trivedi (\cite[Corollary~2]{ SrinivasTrivedi}) showed that the analogue of Question~\ref{worse sing} holds for Hilbert--Samuel multiplicity. They prove that if $(R,\m,k)$ is local then one can always find a small neighborhood $\mathfrak m^N$ such that $\eh (R/(f)) \geq \eh (R/(f+\epsilon))$ for all $\epsilon \in \mathfrak m^N$.

\subsection{Comments on F-pure thresholds}
A conjecture related to the results of this paper is the ACC conjecture for F-pure thresholds, 
very recently resolved by Sato (\cite{Sato},\cite{Sato2}),
which states that if $(R,\m,k)$ is a regular local ring and $f_1,f_2, f_3, \ldots$ elements of $R$ such that
\[
\fpt(f_1) \leq \fpt (f_2) \leq \fpt(f_3) \leq \cdots 
\]
then $\fpt(f_n)=\fpt(f_{n+1})$ for all $n\gg 0$.  
The conjecture is motivated by Shokurov's ACC conjecture for log canonical thresholds (\cite{Shokurov1992}), which was  also resolved in the smooth case (\cite{DFEM, HMX}). As the F-pure threshold function is continuous with respect to the $\m$-adic topology, it follows that if the ACC conjecture holds then for each $f\in R$ there exists $N\in \N$ such that for each $\epsilon\in \m^N$, $\fpt (f+\epsilon)\geq \fpt(f)$. Hern\'andez, N\'u\~nez-Betancourt, and Witt (\cite{HNBW}) recently investigated this particular implication of the ACC conjecture. Their techniques establish that if the Jacobian ideal of an element $f\in R$ is $\m$-primary, then $\fpt (f) = \fpt(g)$ for all $g$ sufficiently close to $f$. Their result is also recovered by Samuel's theorem from (\cite{Samuel1956}) mentioned in the introduction.

A natural analogue to the ACC conjecture for Hilbert--Kunz multiplicity would be a descending chain condition. That is given a sequence of elements $f_1,f_2,f_3,\ldots$ in a regular local ring $(R,\m,k)$ such that $\ehk(R/(f_1))\geq \ehk(R/(f_2))\geq \ehk(R/(f_3))\geq \cdots$ is it necessarily the case that $\ehk(R/(f_n))=\ehk(R/(f_{n+1}))$ for all $n\gg 0$? Work of Monsky (\cite{Monsky1998}) shows the natural analogue of the ACC conjecture for Hilbert--Kunz multiplicity does not hold.

\begin{Example}\label{Monsky example 1}
In \cite{Monsky1998} Monsky computed the Hilbert--Kunz multiplicity of  a family of hypersurfaces of the form  
$R_\alpha = k[[x,y,z]]/(z^4 +xyz^2 + (x^3 + y^3)z + \alpha x^2y^2)$ where $\alpha \in k$,
an algebraically closed field of characteristic $2$.
His computations show that there is a sequence $\alpha_n$ such that $\ehk (R_{\alpha_n}) = 3 + 4^{-n}$,
which gives an infinite decreasing sequence. 
\end{Example}

We suspect there is a family of hyperplanes, similar to that of Example~\ref{Monsky example 1}, whose F-signatures form an infinite increasing sequence.

\specialsection*{Acknowledgements}

We want to thank Ian Aberbach, Bhargav Bhatt, Dale Cutkosky, and Mel Hochster for discussion of this work. In particular, Bhargav Bhatt has helped shape Lemma~\ref{Dis lemma} and Proposition~\ref{bad LC} was suggested by Mel Hochster.
This project started while the first named author was visiting University of Michigan, 
we thank Mel Hochster and University of Michigan for making that trip possible.

We are grateful to Linquan Ma for pointing an inaccuracy in an earlier draft of this paper 
and to Alessandro De Stefani who corrected our mistake by pointing out that Theorem~\ref{Intro theorem}
needs the Gorenstein assumption.

\bibliographystyle{alpha}
\bibliography{References}

\end{document}